\documentclass[11pt]{amsart}

\usepackage{hyperref}
\usepackage{cleveref}
\usepackage[colorinlistoftodos,bordercolor=orange,backgroundcolor=orange!20,linecolor=orange,textsize=scriptsize]{todonotes}

\usepackage{graphicx, enumerate, url}
\usepackage{amssymb,mathtools}

\usepackage{algorithm}
\usepackage{algpseudocode}
\usepackage{color}
\usepackage{tikz}
\usetikzlibrary{3d,calc}

\numberwithin{equation}{section}
\numberwithin{algorithm}{section}

\theoremstyle{plain}
\newtheorem{theorem}{Theorem}[section]
\newtheorem{proposition}[theorem]{Proposition}
\newtheorem{lemma}[theorem]{Lemma}
\newtheorem{corollary}[theorem]{Corollary}

\theoremstyle{definition}
\newtheorem{definition}[theorem]{Definition}

\theoremstyle{remark}

\let\H\undefined

\DeclareMathOperator{\tr}{Tr}
\DeclareMathOperator{\diag}{diag}
\DeclareMathOperator{\rank}{rank}

\newcommand{\N}{\mathbb{N}}
\newcommand{\C}{\mathbb{C}}
\newcommand{\F}{\mathbb{F}}
\newcommand{\H}{\mathbb{H}}
\newcommand{\Q}{\mathbb{Q}}
\newcommand{\R}{\mathbb{R}}

\newcommand{\x}{\mathbf{x}}

\newcommand{\y}{\mathbf{y}}
\newcommand{\z}{\mathbf{z}}
\newcommand{\U}{\mathbf{U}}

\newcommand{\w}{\mathbf{w}}
\newcommand{\W}{\mathbf{W}}
\newcommand{\cA}{\mathcal{A}}

\newcommand{\cE}{\mathcal{E}}
\newcommand{\cF}{\mathcal{F}}

\newcommand{\bbf}{\mathbf{f}}

\newcommand{\bH}{\mathbf{H}}
\newcommand{\bi}{\mathbf{i}}
\newcommand{\bj}{\mathbf{j}}
\newcommand{\bk}{\mathbf{k}}

\newcommand{\bq}{\mathbf{q}}

\newcommand{\bs}{\mathbf{s}}
\newcommand{\bt}{\mathbf{t}}
\newcommand{\bu}{\mathbf{u}}
\newcommand{\bv}{\mathbf{v}}
\newcommand{\bw}{\mathbf{w}}

\newcommand{\rA}{\mathrm{A}}
\newcommand{\rB}{\mathrm{B}}

\newcommand{\rH}{\mathrm{H}}

\newcommand{\rK}{\mathrm{K}}
\newcommand{\rL}{\mathrm{L}}
\newcommand{\rP}{\mathrm{P}}

\newcommand{\rS}{\mathrm{S}}

\newcommand{\rU}{\mathrm{U}}

\newcommand{\0}{\mathbf{0}}

\makeatletter
\@namedef{subjclassname@2020}{\textup{2020} Mathematics Subject Classification}
\makeatother

\begin{document}
\title[Characterizations of the numerical radius and its dual norm]
{On semidefinite programming characterizations of the numerical radius and its dual norm for quaternionic matrices}
\author{Shmuel Friedland}
\address{
 Department of Mathematics, Statistics and Computer Science,
 University of Illinois at Chicago, Chicago, Illinois 60607-7045,
 USA, \texttt{friedlan@uic.edu}
 }
 
  	
\begin{abstract}  
We give a semidefinite programming characterizations of the  numerical radius and its dual norm for quaternionic matrices.  We show that the computation of the numerical radius and its dual norm within $\varepsilon$ precision are polynomially time computable in the data and $|\log \varepsilon |$ using  the short step, primal interior point method.
\end{abstract}
\maketitle

\section{Introduction}\label{sec:intro}
Let $\F$ be either the field real numbers $\R$,  the field of complex numbers, $\C$ or the skew-field of quaternions $\H$:
\begin{equation}\label{deffieldF}
\F\in\{\R,\C,\H\}
\end{equation}
For a positive integer $n$ denote  by $[n]$ the set $\{1,\ldots,n\}$.
For $A=[a_{ij}]\in\F^{m\times n}$ let $A^*=[b_{pq}]\in \F^{n\times m},b_{pq}=\bar a_{qp}, p\in[n],q\in[m]$ be the adjoint matrix.  Identify $\F^n$ with $\F^{n\times 1}$.
 Denote by $\rH_n(\F)$ the real space of selfadjoint matrices $\{A\in\F^{n\times n}, A^*=A\}$.   A selfadjoint matrix $A$ is positive semidefinite (positive definite) if $\x^* A \x\ge 0\, (\x^* A\x >0)$ for $\x\ne \0$,  denoted as $A\succeq 0\,(A\succ 0)$.  We denote by $\rH_{n,+}(\H)$ and $\rH_{n,++}(\H)$ the cone of positive semidefinite matrices and its interior in $\rH_n(\H)$ respectively. 
For $\x\in\F^n$ set $\|\x\|=\sqrt{\x^*\x}$.
Let
\begin{equation}\label{defWAF}
\begin{aligned}
\W(A)=\{\x^* A\x:\x\in \F^n, \|\x\|=1\}, \\
r(A)=\{\max|\x^* A\x|: \x\in \F^n, \|\x\|=1\},
\end{aligned}
\end{equation}
be the numerical range  and the numerical radius of $A\in\F^{n\times n}$ respectively.  For $\F=\R$ the numerical range $\W(A)$ is an interval.  For $\F=\C$ the classical result of Hausdorff-T\"oplitz states that $\W(A)$ is a compact convex set in $\C$.  
Recall the semidefinite programing (SDP) characterization of $r(A)$ stated in \cite[Theorem 1.2]{LO20}, which is essentially due to T. Ando [Lemma 1]\cite{And73}.    (See also  \cite[Theorem $2. 1$]{Mat93}):
\begin{equation}\label{SDPcharCa}
r(A)=\min\{a:\begin{bmatrix} aI_n +Z&A\\A^*& aI_n -Z\end{bmatrix}\succeq 0\}.
\end{equation}

It is shown in Friedland-Li \cite{FL23} that for $A$ whose entries are Gaussian rationals, the above characterization yield that the bit-complexity of approximating $r(A)$ within precision $\varepsilon>0$ is polynomial in the entries of $A$ and $|\log \varepsilon|$.  The aim of this paper is to extend the results of \cite{FL23} to quaternionic matrices.

We now survey briefly the main results of this paper. In $\S$\ref{sec:quatmat} we discuss mostly known properties of quaternionic matrices that are used in this paper.
In $\S$\ref{sec:SDP} we define the SDP for selfadjoint quaternionic matrices.  We show that this SDP problem can be translated to an SDP problem on standard Hermitian matrices.  Hence,  we can adopt the bit-complexity results of de~Klerk-Vallentin \cite{deklerk_vallentin} to quaternions as in \cite{FL23}. In $\S$\ref{sec:numrange} we discuss the numerical range $\W(A)$ and  numerical radius  $r(A)$ of $A\in\H^{n\times n}$.   The main result of this section is Theorem \ref{qnumrrthm}.   Identity \eqref{x*Axiden} gives an explicit expression for a point in $\W(A)$. This expression gives rise to a characterization of $r(A)$ as the maximum of the maximum eigenvalue of $\sum_{l=1}^4 x_l C_l$, where $\sum_{l=1}^4 x_i^2\le 1$ and $C_1,\ldots,C_4$ are certain structured  real symmetric matrices of order $4n$ induced by $A$. In $\S$\ref{sec:SDPchar} we give the SDP characterizations of $r(A)$ and $r^\vee(A)$, where $r^\vee(\cdot)$ is the dual norm of $r(\cdot)$.   In $\S$\ref{sec:pnrange} we introduce the psuedo-numerical range of $A\in\C^{n\times n}$: $\W_{\pi}(A)=\{\x^\top A\x: \x\in\C^n, \|\x\|=1\}$.    The pseudo-numerical range is induced by the quaternionic numerical range.  We show that the pseudo-numerical range has some similar  properties to the quaternionic numerical range.  ($\W_{\pi}(A)$ is not convex.)
\section{Quaternionic matrices}\label{sec:quatmat}
In this section we review some results on quaternionic matrices that we use in this paper.  Most of these results are well known, and can be found in \cite{Lee49, Bre51,Kip51, STZ94}.   For some results that are not mentioned in these paper we give short proofs.
\subsection{Quaternions}\label{subsec:quat}
We denote the elements  of quaternions $\H$ as $q=q_1+q_2\bi+q_2\bj+q_4\bk$, where $q_i\in\R,  i\in[4]$,   and 
\begin{equation*}
\begin{aligned}
\bi^2=\bj^2=\bk^2=-1,\,
\bi\bj=-\bj\bi=\bk, \, \bj\bk=-\bk\bj=\bi,\quad\bk\bi=-\bi\bk=\bj.
\end{aligned}
\end{equation*}
Then
\begin{equation*}
\bar q=q_1-q_2\bi-q_2\bj-q_4\bk,  |q|=\sqrt{\sum_{i=1}^4 q_i^2}, q\bar q=\bar q q=|q|^2,  q^{-1}=|q|^{-2}\bar q \textrm{ for }q\ne 0.
\end{equation*}
Recall that $\overline{a b}=\bar b \bar a$ for $a,b\in\H$.
Let $\Re q=q_1$.  We observe that 
\begin{equation*}
\Re q=\Re \bar q, \quad \Re pq = \Re qp.
\end{equation*}

Denote by $\H^{m\times n}=\{A=[a_{st}], a_{st}, i\in[m],j\in[n]\}$, the set of $m\times n$ quaternionic matrices.   
Thus $\H^{m\times n}$ is a left and right module over $\H$, where $aA=[aa_{ij}]$ and $Aa=[a_{ij}a]$.  We will mostly view $\H^{m\times n}$ as a right module over $\H$, and as a vector space over $\R$.  We identify $\H^{m\times 1}$ and $\H^{1\times n}$ with $\H^m$ and $(\H^n)^\top$ respectively.

Let $A\in \H^{m\times n}$ and $\w\in\H^m$.  One has two representations of $A$  and $\w$ using complex and real numbers
\begin{equation}\label{Arep}
\begin{aligned}
&A=A_1+A_2\bj=(A_{11}+A_{21}\bi)+(A_{12}+A_{22}\bi)\bj, \\ 
&A_1,A_2\in\C^{m\times n},  \, A_{11},A_{21},A_{12},A_{22}\in\R^{m\times n},\\
&\w=\w_1+\w_2\bj=(\w_{11}+\w_{21}\bi)+(\w_{12}+ \w_{22}\bi)\bj,\\
&\w_1,\w_2\in\C^m, \, \w_{11},\w_{21},\w_{12},\w_{22}\in\R^m.
\end{aligned}
\end{equation}
Observe that $\bj A_2=\bar A_2 \bj$.
Denote 
\begin{equation}\label{projHmn}
\begin{aligned}
\rP_1,\rP_2:\H^{m\times n}\to\C^{m\times n},  P_1(A)=A_1, P_2(A)=A_2,  A=A_1+A_2\bj.
\end{aligned}
\end{equation}

View $A\in\mathbb{H}^{m\times n}$ as a linear transformation $\w^\top  \mapsto \w^\top A$ for $\w\in\H^m$.  By letting 
$$\w=\w_1+\w_2\bj=(\w_{11}+\w_{21}\bi)+(\w{12}+\w_{22}\bi)\bj$$
we obtain the complex  and real representation of $A$:  
\begin{equation}\label{crrep}
C(A)=\begin{bmatrix}A_1&A_2\\-\bar A_2&\bar A_1\end{bmatrix},\,R(A)=\begin{bmatrix}A_{11}&A_{21}&A_{12}&A_{22}\\-A_{21}&A_{11}&-A_{22}&A_{12}\\-A_{12}&A_{22}&A_{11}&-A_{21}\\-A_{22}&-A_{12}&A_{21}&A_{11}\end{bmatrix}.
\end{equation}
Observe
\begin{equation*}
\begin{aligned}
&A^\top= A_1^\top + A_2^\top\bj, \, \bar A=\bar A_1-\bj \bar A_2=\bar A_1- A_2\bj, \,A^*=\bar A^\top= A_1^* -A_2^\top \bj,\\
&C(A^\top)=\begin{bmatrix}A_1^\top&A_2^\top\\ -A_2^*&A_1^*\end{bmatrix},
C(A^*)=\begin{bmatrix}A_1^*&-A_2^\top\\ A_2^*&A_1^\top\end{bmatrix}=C(A)^*,  C(\bar A)=\begin{bmatrix}\bar A_1&-A_2\\ \bar A_2&A_1\end{bmatrix},\\
&A B=(A_1+A_2\bj)(B_1+B_2\bj)=A_1B_1-A_2\bar B_2+(A_1B_2+A_2\bar B_1)\bj,\\
&C(AB)=C(A) C(B), \quad A\in \H^{m\times n},  B\in \H^{n\times p}.
\end{aligned}
\end{equation*}
\begin{definition}\label{defQCn}
Denote by $\rm{Q}_{c}^{m\times n}\subset \C^{(2m)\times (2n)}$ the subspace of matrices of the form $C(A)$ given by \eqref{crrep}.
\end{definition}
Note that $\rm{Q}_c^{m\times n}$ is isomorphic to the subset of $m\times n$ matrices whose entries are matrices in $\rm{Q}_c^{1\times 1}$.  

View $q=q_1+q_2\bj, q_1,q_2\in\C$.  Then 
$$Q(q)=\begin{bmatrix}q_1&q_2\\-\bar q_2&\bar q_1\end{bmatrix}, \det Q(q)=|q|^2,  Q(q^{-1})= |q|^{-2}Q(\bar q)\textrm{ for }q\ne 0.$$

Given $a=a_1+a_2\bi+a_3\bj+a_3\bk\in\H$, the similarity class of $C(a)$ corresponds 
\begin{equation*}
\W(a)=\{b\in \H, \, b=\bar q a q,  q\in\H, |q|=1\},
\end{equation*}
which coincides with the numerical range of $a\in\rH^{1\times 1}$.
It is known that 
\begin{equation}\label{numra}
\begin{aligned}
&\W(a)=\{b=b_1+b_2\bi+b_3\bj+b_4\bk, 
\\&b_1=a_1, \sqrt{b_2^2+b_3^2+b_4^2}=\sqrt{a_2^2+a_3^2+b_4^2}\}.
\end{aligned}
\end{equation}
In particular, $\W(a)$ is a convex set in $\H\sim\R^4$, if and only if $a\in\R$.
Note that there exists a unique $a'\in\C, \Im a'\ge 0$ such that $a'\in \W(a)$.
\subsection{Inner product and  the Gram-Schmidt process}\label{subsec:ipGS}
For $A=[a_{st}]\in\H^{n\times n}$ we let $\tr A=\sum_{s=1}^n a_{ss}$ be the trace of $A$.   Clearly, 
\begin{equation*}
\begin{aligned}
&\tr A^\top =\tr A, \quad \Re \tr A=\Re \tr \bar A, \quad   \tr A^*=  \tr \bar A, \\
&\Re\tr FG=\Re \tr GF \textrm{ for } F\in\H^{m\times n}, G\in \H^{n\times m}.
\end{aligned}
\end{equation*}
The inner product on $\H^{m\times n}$, viewed as a right module over $\H$, is defined as 
$$\langle A,B\rangle:=\tr A^*B= \tr (A_1^* B_1+ A_2^\top \bar B_2)+\big(\tr(A_1^*B_2-A_2^\top \bar B_1)\big)\bj,$$
which is formally defined as the inner product on $\mathbb{C}^{m\times n}$.
It satisfies:
\begin{equation*}
\begin{aligned}
&\langle Aa,Bb\rangle=\bar a \langle A, B\rangle b,  \quad \langle B,A\rangle=\overline{\langle A,B\rangle},\\
&|\langle A,B\rangle|\le \sqrt{\langle A, A\rangle} \sqrt{\langle B, B\rangle} \textrm{ Cauchy-Schwarz inequality},\\
&\textrm{equality holds if and only if} Aa=Bb, |a|+|b|>0.
\end{aligned}
\end{equation*}
Then $\| A\|_F:=\sqrt{\langle A,A\rangle}$ is the Frobenius norm.  Note that
\begin{equation*}
\|A\|_F=\|A^\top\|_F=\|\bar A\|_F=\|A^*\|_F= \sqrt{\sum_{s=1}^m\sum_{t=1}^n|a_{st}|^2}.
\end{equation*}
For $\z\in \H^n$ we let $\|\z\|_F=\|\z^\top\|_F=\|\z\|=\|\z^\top\|$.
If we view $\H^{m\times n}$ as a vector space over $\R$, of dimension $4mn$, then 
$$\langle A,B\rangle_{\R}:=\Re \langle A,B\rangle=\Re\tr A^*B$$ is an inner product over $\R$.   Furthermore,
\begin{equation*}
\Re \tr A^*B=\frac{1}{2}\Re \tr C(A)^* C(B).
\end{equation*}

Two vectors $\x,\y \in \H^n$ are called orthogonal if $\langle \x,\y\rangle=0$.
A set of $l$ vectors $\x_1,\ldots,\x_l$ is called orthonormal if $\langle \x_s,\x_t\rangle=\delta_{st}, s,t\in[l]$.   Given $l$ vectors in $\H^n$ one can perform Gram-Schmidt process to obtain $p\le n$ nonzero orthonormal vectors.

Assume that $\bu_1,\ldots, \bu_n$ is a set of $n$ orthonormal vectors in $\H^n$.  Then $\bu_1,\ldots,\bu_n$ is an orthonormal basis in $\H^n$, viewed as a right module over $\H$:
$\x=\sum_{s=1}^n \bu_s \langle \bu_s, \x\rangle$ for each $\x\in\H^n$.  Let $U=[\bu_1\cdots\bu_n]\in \H^{n\times n}$.  Then $\bu_1,\ldots,\bu_n$ is an orthonormal basis in $\H^n$ if and only if $U$ is unitary: $U^*U=I_n$.  Recall that $U$ is unitary if and only $U U^*=I_n$.   

Let $D=\diag(d_1,\ldots,d_n)\in \H^n$.
Then $D$ is unitary if and only if $|d_s|=1,s\in[n]$.   Clearly,  $U$ is unitary if and only if $UD$ is unitary for some diagonal unitary $D$.
Two matrices $A,B\in\H^{n\times n}$ are called unitary similar if $A=U B U^*$ for some unitary $U$.

Assume \eqref{deffieldF}.
Denote by $\rU_n(\F)\subset \F^{n\times n}$ the group of unitary matrices over $\F$.  Observe that $U\in \rU_n(\H)$ if and only $C(U)\in \rU_{2n}(\C)$.

\subsection{Eigenvalues, eigenvectors and the spectral decomposition of normal matrices}\label{subsec:eigsdn}
Let $A\in\H^{n\times n}$.  Then $\lambda\in\H$ is called (right) eigenvalue if there exists an eigenvector $\x\in\H^n\setminus\{\0\}$ such that $A\x=\x\lambda$.  Note that for $q\in\H\setminus\{0\}$ we have the equality $A(\x q)= (\x q) (q^{-1} \lambda q)$.  Hence, an eigenvalue $\lambda$ induces the eigenvalue set $\W(\lambda)$.
This set reduces to $\{\lambda\}$ if and only if $\lambda$ is real.  
Thus, there exists a unique $\lambda'\in \W(\lambda)$ such that $\lambda'\in\C, \Im \lambda'\ge 0$.
It is well known that every $A\in\H^{n\times n}$ has at least one eigenvalue.  Using the Gram-Schmidt process one deduces that $A$ is unitary similar to upper triangular $T$: $A=UTU^*$ for some unitary $U$.  

The matrix $A$ is normal if $AA^*=A^*A$.  Then $A$ is normal if and only if
\begin{equation}\label{canformnorm} 
A=U\diag(\lambda_1,\ldots,\lambda_n)U^*, \quad, U\in\rU_n(\H), \lambda_s\in\C, \Im \lambda_s\ge 0, s\in [n].
\end{equation}   
Clearly, if $A$ of the above form then $A$ is normal.   Assume that $A=V^*TV$, where $V$ is unitary and $T$ is upper triangular,  is normal.  Then $T$ is normal.  Hence $T$ is diagonal.  Choose a diagonal unitary $D$ such that $\bar D T D=\diag(\lambda_1,\ldots,\lambda_n) $, where $\lambda$'s satisfy the conditions of \eqref{canformnorm}.  Let $U=VD$, and deduce  \eqref{canformnorm}.

Observe that $A$ is normal if and only if $C(A)$ is normal.  Furthermore, if $A$ is normal then the eigenvalues of $C(A)$ are $\lambda_1,\bar\lambda_1,\ldots,\lambda_n,\bar\lambda_n$.
The spectral decomposition of a normal $A$ is 
\begin{equation}\label{specnorm}
\begin{aligned}
&A=\sum_{s=1}^n \bu_s \lambda_s \bu_s^*, \quad \lambda_s\in\C, \Im \lambda_s\ge 0, \langle \bu_s,\bu_t\rangle=\delta_{st}, s,t\in[n],\\
&A\bu_s=\bu_s\lambda_s, \quad s\in[n].
\end{aligned}
\end{equation}
Observe that $A$ is unitary if and only if $A$ is normal, and $|\lambda_s|=1$ for $s\in[n]$.    

Recall that $A\in\F^{n\times n}$ is called a self-adjoint if $A^*=A$.   Denote by
\begin{equation}\label{defHn(F)}
\begin{aligned}
\rH_n(\F)=\{A\in \H^{n\times n}: A^*=A\}, \quad \F\in\{\R,\C,\H\}, 
\end{aligned}
\end{equation}
the real subspace of selfadjoint matrices in $\F^{n\times n}$.  Observe that $A\in\H^{n\times n}$ is self-adjoint if and only if $A$ is normal and $\lambda_s\in\R, s\in[n]$.   Equivalently,  $A\in\rH_n(\H)\iff C(A)\in \rH_{2n}(\C)$.  
Assume that $A\in \rH_n(\H)$.   Then the eigenvalues of $A$ and $C(A)$ are arranged in a nondecreasing order:
\begin{equation}\label{eigsa}
\begin{aligned}
&\lambda_{\max}(A)=\lambda_1(A)\ge \ldots\ge \lambda_n(A)=\lambda_{\min}(A),\\
&\lambda_{2l-1}(C(A))=\lambda_{2l}(C(A))=\lambda_l(A), \quad l\in[n].
\end{aligned}
\end{equation}
The eigenvalues of $A$ have the maximum and the minimum characterization of the Rayleigh quotient
\begin{equation*}
\lambda_{\max}=\max_{\x\ne \0} \frac{\x^* A\x}{\x^*\x},  \quad \lambda_{\min}=\min_{\x\ne \0} \frac{\x^* A\x}{\x^*\x}.
\end{equation*}
The other eigenvalues have $\max-\min$ as in \cite[Section 4.4]{Fri15}.

A sefadjoint matrix $A$ is positive semidefinite (positive definite) if $\lambda_{\min}(A)\ge 0\, (\lambda_{\min}(A) >0)$.   This is equivalent to $\x^* A \x\ge 0\, (\x^* A\x >0)$ for $\x\ne \0$. We denote by $\rH_{n,+}(\F)$ and $\rH_{n,++}(\F)$ the cone of positive semidefinite matrices and its interior in $\rH_n(\F)$ respectively. 

We view $\rH_n(\H)$ as a real vector space of dimension $n(2n-1)$ with an inner product $\langle A,B\rangle =\Re \tr AB$.  Observe that $\rH_{n,+}(\H)$ is a selfadjoint cone in $\rH_n(\H)$ with respect to the above inner product.  That is, $\Re\tr AB\ge 0$ for all $B\in\rH_{n,+}(\H)$ if and only if $A\in \rH_{n,+}(\H)$.
\subsection{Singular value decomposition}\label{subsec: svd}
Let $A\in \H^{m\times n}$.  
Then $F=A^* A\in \rH_{n,+}(\H), G=A A^*\in\rH_{m,+}(\H)$.   We denote by $\sigma_1^2(A)\ge \cdots\ge  \sigma_r^2(A)>0$ all positive eigenvalues $F$.   Here $r$ is the rank of $A$.  
We agree that $\sigma_l(A)=0$ for $l>r$.   Let $F\z_s=\z_s\sigma_s^2(A)=\sigma_s^2(A)\z_s$, where $\z_1,\ldots,\z_n$ are orthonormal.  Hence $G(A\z_s)=\sigma_s^2(A)(A\z_s)$ for $s\in[n]$.   
Let $\w_s=\sigma_s^{-1}(A)(A\z_s)$  for $s\in[r]$.  Then $\w_1,\ldots,\w_r\in\H^m$ are orthonormal vectors.  Hence $G$ has exactly $r$ positive eigenvalues $\sigma_1^2(A)\ge \cdots\ge  \sigma_r^2(A)>0$.  The singular value decomposition (SVD) of $A$ is
\begin{equation}\label{SVDdec}
\begin{aligned}
&A=\sum_{l=1}^r \sigma_l(A)\w_l\z_l^*= W_r\Sigma_r Z_r^*, \\&W_r=
[\w_1\cdots \w_r], \Sigma_r=\diag(\sigma_1(A),\ldots,\sigma_r(A)), Z_r=[\z_1\cdots\z_r].
\end{aligned}
\end{equation}
Recall that 
$C(A^*A)=C(A^*) C(A)=C(A)^*C(A)$.  Hence the number of positive singular values of $C(A)$ is $2r$ and they satisfy the equalities 
\begin{equation}\label{svalhatA}
\begin{aligned}
\sigma_{2s-1}(C(A))=\sigma_{2s}(C(A))=\sigma_s(A), \quad s\in [r].
\end{aligned}
\end{equation}

The following result is a straightforward consequence of the SVD decomposition \eqref{SVDdec} as for the complex matrices \cite[Theorem 4.11.1]{Fri15}:
\begin{proposition}\label{SAA}  Let $A\in\H^{m\times n}$ then $H(A):=\begin{bmatrix}0&A\\A^*&0\end{bmatrix}\in \rH_{m+n}(\H)$.
Its nonzero eigenvalues are $\pm\sigma_1(A),\ldots,\pm \sigma_r(A)$.
\end{proposition} 
\subsection{Norms on $\H^{m\times n}$}\label{subsec:norms}
\begin{definition}\label{defnorm}  A map $\|\cdot\|:\mathbb{H}^{m\times n}\to [0,\infty)$ 
is called an $\R$-norm if the following conditions hold:
\begin{equation*}
\begin{aligned}
&\|A\|=0 \iff  A=0,\\
&\|A+B\|\le \|A\|+\|B\| \quad \rm{ triangle\, inequality},\\
&\|Aa\|=\|A\||a| \,\rm{for}\, a\in\R \quad (\R-\rm{homogeneity}),\\
\end{aligned}
\end{equation*}
An $\R$-norm is called an $\H$-norm if one has the equality $\|Aa\|=\|aA\|=\|A\||a|$ for $a\in\H$.
\end{definition}
\begin{proposition}\label{Schatnorm}  Let $A\in\rH^{m\times n}$ with $\rank A=r$.
Then for $p\in[1,\infty]$ the quantity $\|A\|_p:=\big(\sum_{l=1}^r\sigma_l^p(A)\big)^{1/p}$ is an $\H$-norm on $\H^{m\times n}$, (called $p$-Schatten norm).  Furthermore,
\begin{equation}\label{UVSchateq}
\|UAV\|_p=\|A\|_p \,\rm{for}\, U\in \rU_m(\H),  V\in \rU_n(\H), p\in[1,\infty].
\end{equation}
\end{proposition}
\begin{proof} 
From the definition of the SVD of $A$ we easily deduce.
\begin{equation}
\sigma_l(Aa)=\sigma_l(aA)=|a|\sigma_l(A), \quad l\in[r], 
\sigma(\bar A
\end{equation}
Hence $\|Aa\|_p=\|aA\|_p=|a|\|A\|_p$.  Clearly, $\|A\|_p=0$ if and only if $A=0$.
The equality \eqref{svalhatA} yields that $\|A\|_p=2^{-1/p}\|C(A)\|_p$.  It is well known that $\|C\|_p$ is a norm on $\C^{m\times n}$ \cite[Problem 4, Section 4.11]{Fri15}.   Hence $\|\cdot\|_p$ satisfies the triangle inequality.  

Clearly, $UAV$ and $A$ have the same singular values for unitary $U$ and $V$.  Hence, \eqref{UVSchateq} holds.
\end{proof}

Note that 
\begin{equation}\label{21intynrm}
\begin{aligned}
&\|A\|_F=\|A\|_2,\\
&\|A\|_1=\sum_{l=1}^n \sigma_l(A)-\textrm{the nuclear norm},\\
&\|A\|_{\infty}=\sigma_1(A)=\max\{|\langle \bw,A\z\rangle|, \|\bw\|=\|\z\|=1\}=\\
&\max\{\Re \langle \bw,A\z\rangle, \|\bw\|=\|\z\|=1\}-\textrm{the spectral norm}. 
\end{aligned}
\end{equation}
We now recall the definition of the dual norm on $\H^{m\times n}$:
\begin{definition}\label{defdnrm}
Let $\|\cdot\|$ be an $\R$-norm on $\mathbb{H}^{m\times n}$.  Then 
\begin{equation}\label{defdnrm1}
\|A\|^\vee:=\max\{\Re\langle B,A\rangle, \|B\|\le 1\}
\end{equation}
is the dual $\R$-norm on $\H^{m\times n}$.
\end{definition}
Recall that the dual of the dual norm is the original norm.  Assume that $\|\cdot\|$ is an 
$\H$-norm.  Then 
$\|A\|^\vee:=\max\{|\langle B,A\rangle|, \|B\|\le 1\}$.  Hence
$\|\cdot\|^\vee$ is an $\H$-norm.  It is straightforward to show that $\|A\|_p^\vee=\|A\|_q$, where $1/p+1/q=1$.
\section{Semidefinite programming for quaternionic matrices}\label{sec:SDP}
Assume \eqref{defHn(F)}.  Then, $\rH_n(\C)=\rH_n$, and $\rH_n(\R)=\rS_n(\R)$-the space of real symmetric matrices of order $n$.
Clearly, $\rH_n(\R)\sim\R^{n(n+1)/2}$,  $\rH_n\sim\R^{n^2}$, and $\rH_n(\H)\sim\R^{n(2n-1)}$.
The inner product in $\rH_n(\F)$ is $\langle A,B\rangle =\Re\tr AB$.  
Note that  $\tr AB\in\R$ for $A,B\in \rH_n(\F)$ and $\F\in\{\R,\C\}$.
However,  $\tr AB$ may not be a real number for $n\ge 2$, $A,B\in\rH_n(\H)$ :
\begin{equation*}
A=\begin{bmatrix}0&\bi\\-\bi&0\end{bmatrix}, B=\begin{bmatrix}0&\bj\\-\bj&0\end{bmatrix}, \tr AB=-2\bi\bj=-2\bk.
\end{equation*}

Thus $\|A\|_F=\sqrt{\langle A, A\rangle}=\sqrt{\tr A^2}$ is the Frobenius norm of $A\in\rH_n(\F)$.  
For $Y_0\in \rH_n(\F)$ and $r\ge 0$ denote by $\rB(Y_0,r)=\{X\in\rH_n(\F),  \|X-Y_0\|_F\le r\}$ the closed ball in $\rH_n(\F)$ centered at $Y_0$ with radius $r$.

A standard semidefinite program  for $\F\in \{\R,\C\}$ is
\begin{equation}\label{srSDP}
val=\inf\{\langle C,X\rangle: X\in \rH_{n,+}(\F), \langle A_j, X\rangle =b_j, j\in [m]\},
\end{equation}
where $C,A_j\in \rH_n(\F), b_j\in\R$ for $j\in[m]$.
We also will call the above infimum problem as a standard semidefinite program  for quaternions: $\F=\H$.
Denote by $\cF$ the feasible set
\begin{equation}\label{defcF}
\cF=\{X\in \rH_{n,+}(\F):\langle A_j, X\rangle =b_j, j\in [m]\}.
\end{equation}
Let
\begin{equation}\label{defLAb}
\rL(A_1,\ldots,A_m,b_1,\ldots,b_m)=\{X\in\rH_n(\F): \langle A_j, X\rangle =b_j, j\in [m]\}.
\end{equation}
Then $\rL(A_1,\ldots,A_m,b_1,\ldots,b_m)$ is an affine subspace whose dimension is
$k\in\{-1,0\}\cup[\dim \rH_n(\F)]$.  So $k=-1$ if and only if $\rL(A_1,\ldots,A_m,b_1,\ldots,b_m)=\emptyset$, and $k\ge 0$ if $\rL(A_1,\ldots,A_m,b_1,\ldots,b_m)=X_0+\U$ where $$\U=\rL(A_1,\ldots,A_m,0,\ldots,0)\subset \rH_n(\F)$$ is a subspace of dimension $k$.  

By introducing a standard basis in $\rH_n(\F)$ one can use real Gauss elimination to determine
the dimension $d$ of $\rL(A_1,\ldots,A_m,b_1,\ldots,b_m)$.  In particular, if $d=\dim \rH_n(\F)-\delta\ge 0$, then $m\ge \delta$, and there exists a subset $\{b_{i_1},\ldots, b_{i_{\delta}}\}$ for some $\{1\le i_1<\ldots <i_{\delta}\}\subset [m]$ such that
\begin{equation}\label{recondL}
\begin{aligned}
&\rL(A_1,\ldots,A_m,b_1,\ldots,b_m)=\rL(A_{i_1},\ldots,A_{i_\delta},b_{i_1},\ldots,b_{i_\delta}),  \\
&\dim \rL(A_1,\ldots,A_m,b_1,\ldots,b_m)=d=\dim \rH_n(\F)-\delta\ge 0.
\end{aligned}
\end{equation}

As explained in \cite{FL23} we can assume that a standard SDP problem is of the form \cite[Eq. (1)]{VB96}:
\begin{equation}\label{SDPsf}
\begin{aligned}
&val=\inf\{\sum_{i=1}^k c_i s_i: X_0+\sum_{i=1}^k s_i X_i\succeq 0, \\
&X_0,\ldots,X_k\in\rH_n(\F),(s_1,\ldots,s_k)^\top\in\R^k\}.
\end{aligned}
\end{equation}
Without loss of generality we can assume that $X_1,\ldots,X_k$ are linearly independent,  and either $X_0=0$ or $X_0,X_1,\ldots,X_k$ are linearly independent.
In that case there is a simple way to characterize the set
\begin{equation}\label{desccA}
\begin{aligned}
&\cA(X_0,\ldots,X_k)=\{X=X_0+\sum_{i=1}^k s_iX_i:\\
&(s_1,\ldots,s_k)^\top \in\R^k\},  
\end{aligned}
\end{equation}
where $X_0,\ldots,X_k\in\rH_n(\F)$.   
\begin{lemma}\label{lcharcAX}  Let $X_0,X_1,\ldots,X_k\in\rH_n(\F)$, where $1\le k<\dim\rH_n(\F)$.  Assume that $X_1,\ldots,X_k$ are linearly independent.  Set  $m=\dim\rH_n(\F)-k$.  The set \eqref{desccA} is given by  \eqref{defLAb} as follows:
\begin{enumerate}[(a)]
\item Assume that $X_0=0$.  Then,  $b_i=0$ for $i\in[m]$, and
$A_1,\ldots,A_m$ is a basis in the subspace 
$\{X: \Re\tr X_i X=0, i\in[k]\}$.
\item Assume that $X_0,\ldots,X_k$ are linearly independent.   Then $b_i=0$ for $i\in[m-1],$ and matrices $A_1,\ldots, A_{m-1}$ is a basis in the subspace 
$\{X\in\rH_n(\H), \Re\tr X_i X=0, i=0,1,\ldots,k\}$.  A matrix $A_m$ is  a solution to $\{X\in\rH_n(\H), \tr X_0A_m=b_m=1, \tr X_iA_m=0, i\in[k]\}$.
\end{enumerate}
\end{lemma}
The proof of the Lemma is straightforward.

Then the dual problem is of the form \cite[Eq. (27)]{VB96}:
\begin{equation}\label{dSDPsf}
val^\vee=\sup\{-\Re\tr X_0 Z: \Re\tr X_i Z=c_i, i\in[k],Z\in \rH_n(\F)\}.
\end{equation}
The Slater constraint condition \cite[Theorem 4.7.1]{GM12}, see also \cite[Corollary 2.2]{FL23}, is:
\begin{theorem}\label{corSlater}  Assume that the feasible set of \eqref{SDPsf} contains a positive definite matrix, and $val$ is finite.  Then the dual problem \eqref{dSDPsf} is feasible, and $val=val^\vee$.
\end{theorem}
\subsection{Complexity results for semidefinite programming}\label{subsec:comSDP}
As in \cite{FL23} it is possible to adopt the complexity results of  de Klerk-Vallentin  \cite[Theorem 1.1]{deklerk_vallentin} to quaternions.   Namely, we translate the SDP problem \eqref{srSDP} for $\F=\H$ to the SDP problem  \eqref{srSDP} for $\F=\C$, by considering the matrices $C(A)\in\rH_n$ for $A\in\rH_n(\H)$.

Denote by $\Q[\F]\subset \F$ the subfield of rationals over $\F$.  Thus $\Q[\R]$ is the field of real rationals,  $\Q[\C]$ is the field of Gaussian rationals: $\Q+\Q\bi$, and $\Q[\H]$ is the filed of quaternionic rationals: $\Q+\Q\bi+\Q\bj+\Q\bk$.
Then the complexity results of de Klerk-Vallentin can be stated in the following form \cite[Section 2.3]{FL23}:
\begin{theorem}\label{cdeKVal}
Let $\F$ be either the field of real numbers $\R$, the field of complex numbers $\C$, or the skew-field of quaternions $\H$.   Consider the SDP problem \eqref{srSDP}. 
Assume that $A_j\in \rH_n(\F)\cap \Q^{n\times n}[\F], b_j\in \Q$ for $j\in [m]$,  and  $\dim\rL(A_1,\ldots,A_m,b_1,\ldots,b_m)=n^2-m\ge 1$.
Suppose that there exists $Y_0\in\rH_{n,+}(\F)\cap \Q^{n\times n}[\F]$ in the feasible set $\cF$ given by \eqref{defcF}, and $0<r\le R\in\Q$ such that the condition
\begin{equation}\label{BrRcond}
\begin{aligned}
&\rL(A_1,\ldots,A_m,b_1,\ldots,b_m)\cap \rB(Y_0,r)\subseteq \cF\\
&\subseteq \rL(A_1,\ldots,A_m,b_1,\ldots,b_m)\cap \rB(Y_0,R)
\end{aligned}
\end{equation} 
holds.
Then for $C\in\rH_n(\F)\cap \Q^{n\times n}[\F]$ and rational $\varepsilon>0$ one can find $X^\star\in \cF$ in poly-time using the short step primal interior point method combined with  Diophantine approximation such that: $\langle C,X^\star\rangle-\varepsilon\le val$, where the polynomial is in $n,\log R/r, |\log \varepsilon|$ and the bit size of the data $Y_0,C,A_1,\ldots,A_m,b_1,\ldots,b_m$.
\end{theorem}

\section{Numerical range}\label{sec:numrange}
The numerical range  and the numerical radius of $A\in \H^{n\times n}$, referred sometimes as qnumerical range and qnumerical radius, is given by 
\begin{equation}\label{defWArA}
\begin{aligned}
&\W(A)=\{\x^* A\x: \x\in \H^n, \|\x\|=1\},\\
&r(A)=\max\{|\x^* A\x|:\|\x\|=1\}=\max\{\Re \bar t \x^* A \x :\|\x\|=1, |t|=1\}.
\end{aligned}
\end{equation}
We denote by co$(\W(A))$ the convex hull of $\W(A)$ in $\H\sim\R^4$.   As $\W(A)$ is a compact set, it follows that co$(\W(A))$ is a compact convex set.

Recall that for $n=1$ the numerical range $\W(a)$ is given by \eqref{numra}.  Hence, it is convex set if and only if $a\in\R$.  Observe that 
\begin{equation*}
\begin{aligned}
&\rm{co}(\W(a))=\{\x=(x_1,x_2,x_3,x_4)^\top\in\R^4:\\
&x_1=a_1, \|(x_2,x_3,x_4)^\top\|\le  \|(a_2,a_3,a_4)^\top\|\}.
\end{aligned}
\end{equation*}

Denote by
\begin{equation}\label{defsymskew}
\begin{aligned}
&\rS_n(\F)=\{A\in\F^{n\times n}: A^\top =A\},\\
&\rA_n(\F)=\{A\in\F^{n\times n}: A^\top =-A\}.
\end{aligned}
\end{equation}
We now show that co$(\W(A))$ and $r(A)$ for $A\in\H^{n\times n}$ have some similar characterizations to $\W(A)$ and $r(A)$ for $A\in\C^{n\times n}$ as in \cite{FL23}.  
\begin{theorem}\label{qnumrrthm}   Let $A\in\H^{n\times n}$,  and write $A=A_1+A_2 \bj$, where $A_1, A_2\in\C^{n\times n}$.  Define the following matrices of order $n$, $2n$ and $4n$ respectively.
\begin{equation}\label{defC1234}
\begin{aligned}
&A_{1}=A_{11}+A_{21}\bi, A_{l1}=E_{l}+F_{l}\bi\in \rH_n(\C),  E_{l}\in\rS_n(\R), F_{l}\in \rA_n(\R),\\
&l\in[2], S=\frac{1}{2}(A_2+A_2^\top)\in \rS_n(\C),  T=\frac{1}{2}(A_2-A_2^\top)\in \rA_n(C),\\
&S=S_1+S_2\bi, S_1,S_2\in\rS_n(\R), T=T_1+T_2\bi, T_1,T_2\in \rA_n(\R),\\
&C(A)=B_{11}+B_{21}\bi,  B_{11},B_{21}\in \rH_{2n}(\C), \\ 
&B_{11}=\begin{bmatrix}E_1+F_1\bi&T_1 +T_2\bi\\-T_1+T_2\bi&E_1-F_1\bi\end{bmatrix}
,B_{21}=\begin{bmatrix}E_2+F_2\bi&S_2-S_1\bi\\S_2+S_1\bi&E_2-F_2\bi\end{bmatrix},\\
&C(-A\bj)=B_{12}+B_{22}, B_{12}\in\rS_n(\C), B_{22}\in\rA_n(\C),\\
&B_{12}=\begin{bmatrix}S_1+S_2\bi&F_2-E_2\bi\\-F_2-E_2\bi&S_1-S_2\bi\end{bmatrix},\\
&C_1=\begin{bmatrix}E_1&T_1&-F_1&-T_2\\-T_1&E_1&-T_2&F_1\\F_1&T_2&E_1&T_1\\T_2&-F_1&-T_1&E_1\end{bmatrix},C_2=\begin{bmatrix}E_2&S_2&-F_2&S_1\\S_2&E_2&-S_1&F_2\\F_2&-S_1&E_2&S_2\\S_1&-F_2&S_2&E_2\end{bmatrix},\\
&C_3=\begin{bmatrix}S_1&F_2&S_2&-E_2\\-F_2&S_1&-E_2&-S_2\\S_2&-E_2&-S_1&-F_2\\-E_2&-S_2&F_2&-S_1\end{bmatrix},C_4=\begin{bmatrix}S_2&-E_2&-S_1&-F_2\\-E_2&-S_2&F_2&-S_1\\-S_1&-F_2&-S_2&E_2\\F_2&-S_1&E_2&S_2\end{bmatrix}.
\end{aligned}
\end{equation}
Then
\begin{enumerate}[(a)]
\item The matrices $C_1,C_2,C_3,C_4$ of order $4n$ are real symmetric.
\item  Let 
$$t=t_1+t_2\bi+t_3\bj+t_4\bk, \,\bt=(t_1,t_2,t_3,t_4)^\top\in\R^4, \,\|\bt\|=1$$
be fixed.
The two supporting hyperplanes of $\rm{co}(\W(C))\subset \H$  of the form
$\Re\bar t  w=Const$ are 
\begin{equation*}
\Re\bar t  w=\lambda_{\min}(\sum_{l=1}^4 t_l C_l), \quad \Re\bar t  w=\lambda_{\max}(\sum_{l=1}^4 t_l C_l).
\end{equation*}
That is, every $w=w_1+w_2\bi+w_3\bj+w_4\bk\in \rm{co}(\W(A))$ satisfies the sharp inequalities
\begin{equation}\label{sprangeinq}
\lambda_{\min}(\sum_{l=1}^4 t_lC_l)\le \Re \bar t w\le \lambda_{\max}(\sum_{l=1}^4 t_lC_l).
\end{equation}
\item The numerical radius of $A$ is given by the formula
\begin{equation}\label{numrform}
\begin{aligned}
r(A)=\max_{\|\bt\|=1}\lambda_{\max}(\sum_{l=1}^4 t_l C_l)=\max_{\|\bt\|\le 1}\lambda_{\max}(\sum_{l=1}^4 t_l C_l).
\end{aligned}
\end{equation}
\end{enumerate}
\end{theorem}
\begin{proof} (a) From the definitions of $E_l,F_i, S_l,T_l$ it follows straightforward that $C_1,C_2,C_3,C_4\in\rS_{4n}(\R)$. 

\noindent
(b) Let  $\z\in \H^n$ and write $\z=\bu+\bv\bj, \bu,\bv \in \C^n$.  Note that $\|\z\|^2=\|\bu\|^2+\|\bv\|^2$. 
Then 
\begin{equation*}
\z^\top A=(\bu^\top+\bv^\top\bj)(A_1+A_2\bj)=(\bu^\top A_1-\bv^\top \bar A_2)+(\bu^\top A_2  +\bv^\top \bar A_1)\bj,
\end{equation*}
and
\begin{equation}\label{x*Axiden}
\begin{aligned}
&\z^\top A \bar \z=\big((\bu^\top A_1-\bv^\top \bar A_2)+(\bu^\top A_2  +\bv^\top \bar A_1)\bj\big)(\bar \bu-\bv\bj\big)=\\
&(\bu^\top A_1-\bv^\top \bar A_2)\bar \bu+(\bu^\top A_2+\bv^\top \bar A_1)\bar\bv+\\
&\big((\bu^\top A_2+\bv^\top \bar A_1)\bu+(-\bu^\top\ A_1+\bv^\top \bar A_2)\bv\big)\bj=\\
&[\bu^\top\bv^\top]C(A)\begin{bmatrix}\bar \bu\\\bar\bv\end{bmatrix} + [\bu^\top\bv^\top]C(-A\bj)\begin{bmatrix}\bu\\\bv\end{bmatrix}\bj
\end{aligned}
\end{equation}

Observe that
\begin{equation*}
\begin{aligned}
&[\bu^\top \bv^\top ]C(A)\begin{bmatrix}\bar \bu\\\bar\bv\end{bmatrix}=
[\bu^\top\bv^\top]\big(B_{11}+ B_{21}\bi\big)\begin{bmatrix}\bar\bu\\\bar\bv\end{bmatrix},\\
&[\bu^\top\bv^\top]C(-A\bj)\begin{bmatrix}\bu\\\bv\end{bmatrix}=[\bu^\top\bv^\top](B_{12}+B_{22}\b) \begin{bmatrix}\bu\\\bv\end{bmatrix}=
[\bu^\top\bv^\top] B_{21} \begin{bmatrix}\bu\\\bv\end{bmatrix}.
\end{aligned}
\end{equation*}
Set $\begin{bmatrix}\bu\\\bv\end{bmatrix}=\bs -\bt\bi\in\C^{2n}$,  where $\bs,\bt\in\R^{2n}$.
Hence
\begin{equation}\label{z*Czfirm}
\begin{aligned}
\z^\top A\bar \z=[\bs^\top\bt^\top]\big(C_1+ C_2\bi+ C_3\bj+ C_4\bk\big)\begin{bmatrix}\bs\\\bt\end{bmatrix},
\end{aligned}
\end{equation}
and
\begin{equation}\label{Retz*Czid}
\Re\bar t \z^\top C\bar \z=[\bs^\top\bt^\top]\big(\sum_{l=1}^4 t_lC_l\big)\begin{bmatrix}\bs\\\bt\end{bmatrix}.
\end{equation}
The above equality yields \eqref{sprangeinq}.

\noindent (c) Taking maximum and minimum on $[\bs^\top\bt^\top]$ with norm one in 
\eqref{Retz*Czid} we deduce the sharp inequalities and $t, |t|=1$ we obtain the characterization \eqref{numrform}.
\end{proof}
\begin{proposition}\label{propqradprop} The numerical radius is an $\R$-norm on $\rH^{n\times n}$. 
Furthermore, for $A\in\H^{n\times n}$ and $U\in\rU_n(\H)$ the following conditions hold:
\begin{enumerate}[(a)]
\item $\W(U^*AU)=\W(A)$ and $r(U^*AU)=r(A)$.
\item $\W(A^*)=\overline{\W(A)}$ and $r(A^*)=r(A)$.
\item The following sharp inequalities hold
\begin{equation}\label{qradnormin}
\frac{1}{2}\|A\|_\infty \le r(A)\le \|A\|_{\infty}.
\end{equation}
\end{enumerate}
\end{proposition}
\begin{proof}
Assume that $\z^\top A\bar \z=0$ for all $\z$ of norm one.    The equalities \eqref{x*Axiden} yield $C(A)=0$.  Hence $A=0$. Clearly, for $a\in\R$ one has the equality $r(Aa)=|a|r(A)$.   The maximal characterization of $r(A)$ \eqref{defWArA} yields
$r(A+B)\le r(A)+r(B)$.   Hence, $r(\cdot)$ is and $\R$-norm.
 
 \noindent (a)
Clearly, $\x^*(U^* A U)\x=(U\x)^* A (U\x)$.  Hence,  $\W(U^*AU)=\W(A)$ and $r(U^*AU)=r(A)$.

\noindent (b)
Observe that $\overline{\x^* A\x}= \x^* A^* \x$.  Hence $\W(A^*)=\overline{ \W(A)}$
and $r(A)=r(A^*)$.   

\noindent (c)  Recall that the sharp inequality \eqref{qradnormin} holds for $A\in\C^{n\times n}$ \cite[(5.7.23)]{HJ13}.  The characterizations of $\|A\|_{\infty}=\sigma_1(A)$ \eqref{21intynrm} and $r(A)$ and \eqref{defWArA} yield the inequality $r(A)\le \|A\|_{\infty}$.   The equialities \eqref{svalhatA} imply that $\sigma_1(A)=\sigma_1(C(A))$.  The equalities \eqref{x*Axiden} yield 
\begin{equation*}
\begin{aligned}
&|\z^\top A\z|=\sqrt{|[\bu^\top\bv^\top]C(A)\begin{bmatrix}\bar \bu\\\bar\bv\end{bmatrix}|^2 + |[\bu^\top\bv^\top]C(-A\bj)\begin{bmatrix}\bu\\\bv\end{bmatrix}|^2}\Rightarrow \\
&r(A)\ge r(C(A))\ge\frac{1}{2} \sigma_1(C(A))=\frac{1}{2}\|A\|_{\infty}.
\end{aligned}
\end{equation*}
\end{proof}

Recall that the numerical range of a normal complex  matrix is a convex hull of its eigenvalues.  The corresponding result for quaternionic normal matrices is:
\begin{proposition}\label{qrnmat}  
\begin{enumerate}[(a)]
\item Assume that $A\in \H^{n\times n}$ is normal, with the eigenvalues $\lambda_1,\ldots,\lambda_n\in\C$.  Then $\W(A)$ is a union of convex combinations of $q_1,\ldots,q_n$, where $q_i\in \W(\lambda_i)$ for $i\in[n]$.  Hence, $\rm{co}(\W(A))$ is a convex hull of $\cup_{i=1}^n \W(\lambda_i)$.
\item Assume that $A\in\rH_n(\H)$.  Then $\W(A)$ is an interval $[\lambda_{\min}(A),\lambda_{\max}(A)]$.  Furthermore,
\begin{equation}\label{SAxid}
\bar t \x^* A\x t= |t|^2 \x^* A\x, \quad \x\in\H^n, t\in\H. 
\end{equation}
\end{enumerate}
\end{proposition}
\begin{proof} (a)  In view of part (a) of Proposition \ref{propqradprop} we can assume that $A=\diag(\lambda_1,\ldots,\lambda_n)$.  Let $\x=(x_1,\ldots,x_n)^\top=(t_1|x_1|,\ldots,t_n|x_n|)^\top$, where $t_i\in\H, |t_i|=1,i\in[n]$ and $\sum_{i=1}|x_i|^2=1$.
Clearly, $q_i=\bar t_i \lambda t_i\in \W(\lambda_i)$ for $i\in[n]$.  
Then $\x^* A\x=\sum_{i=1}^n |x_i|^2 \bar t_i\lambda_i t_i$ is a convex combination $q_1,\ldots,q_n$.  Vice versa,  any convex combination of $q_i\in \W(\lambda_i)$ for $i\in[n]$ is of the form $\x^* A\x$ for a corresponding $\x\in\H^n, \|\x\|=1$.
Clearly, co$(\W(A))$ is the convex hull of $\cup_{i=1}^n \W(\lambda_i)$.

\noindent
(b) Recall that if $A\in\rH_n(\H)$ then $A$ is normal with $\lambda_i\in\R$.  Hence $\x^*A\x$ is a convex combination of $\lambda_1,\ldots\lambda_n$, and $\W(A)=[\lambda_{\min}(A),\lambda_{\max}(A)]$.  As $\x^*A\x\in\R$ we deduce \eqref{SAxid}.
\end{proof}
\section{The SDP characterizations of qradius and its dual norm}\label{sec:SDPchar}
\subsection{Characterizations of $r^\vee(\cdot)$}\label{subsec:charrvee}
The definition of the dual norm \ref{defdnrm},  and the fact the the dual of the dual norm is the original norm, yields the following characterizations of the dual norm of the qradius and the norm qradius:
\begin{equation}\label{defrvee}
\begin{aligned}
&r^\vee(C)=\max_{r(A)\le 1}\Re\tr A^* C, \textrm{ for } C\in\H^{n\times n},\\
&r(A)=\max_{r^\vee(C)\le 1} \Re C^* A.
\end{aligned}
\end{equation}
\begin{proposition}\label{rveeub}  
The set of the extreme points of the unit ball of the $\R$-norm  $r^\vee(\cdot)$ on $\H^{n\times n}$ is 
\begin{equation}\label{extpt}
\cE=\{\z  t\z^*, \z\in\H^n: \|\z\|=1, t\in \H, |t|=1\}.
\end{equation}
Hence, The norm $r^\vee(\cdot)$ is invariant under the unitary similarity.
\end{proposition}
\begin{proof}  Observe that 
$\Re \bar t \x^*A\x=\Re \tr (\x t \x^*)^*A$.
Compare the definition of $r(A)$ in \eqref{defWArA} with the second equality in  \eqref{defrvee} to deduce that the set of the extreme points of the unit ball of the norm $r^\vee(\cdot)$ is a subset of $\cE$.  Note that $\cE$ is a subset of the unit sphere $\rK(\H^{n\times n})=\{A\in\H^{n\times n}: \|A\|_F=1\}$, which is the set of the extreme points of the unit ball in $\H^{n\times n}$ with respect ot the norm $\|\cdot \|_F$.  Hence, $\cE$ is the set of the extreme points of the unit ball of the norm $r^\vee(\cdot)$.

Clearly $U\cE U^*=\cE$ for each $U\in\rU_n(\H)$.  Hence $r^\vee(\cdot)$ is invariant  
under the unitary similarity.
\end{proof}
\begin{theorem}\label{rveeSDPchar}  Let $Y\in \H^{n\times n}$.  Then $r^\vee(Y)\le 1$ if and only if there exists $Z=\begin{bmatrix}X&Y\\Y^*&X\end{bmatrix}\in\rH_{2n,+}(\H)$ such that $\tr X=1$.  
\end{theorem}
\begin{proof} Assume that $Y=\z t\z^*$, where $\z\in \H^n, \|\z\|=1,|t|=1$.  Then $Z=\begin{bmatrix}\z\\ \z\bar t\end{bmatrix}[\z^*\, t\z^*]$.    As the set of the extreme points of the unit ball of $r^\vee(C)\le 1$ is $\cE$ we deduce that there exists $Z$ of the above form.  

Suppose that $Z$ of the above form is positive semidefinite, and $\tr X=1$.
Let us assume first that $X=\frac{1}{n} I_n$.  Then $Z=\frac{1}{n}I_{2n}+ H(Y)$, where $H(Y)$ is given in Proposition \ref{SAA} .
As $Z\succeq 0$, Proposition \ref{SAA} yields that  $\sigma_1(Y)\le 1/n$.  Let $F=nY$.  Then $\sigma_1(F)\le 1$.  We claim that $F$ is a convex combination of $n+1$ unitary matrices.   Let $F=W\Sigma_n(F)V^*$ be the SVD decomposition of $F$,
where  $W,V$ are unitary matrices.  Let $\bbf=(\sigma_1(F),\ldots,\sigma_n(F))^\top\in\R^{n}$.  Then $\|\bbf\|_\infty\le 1$.  Recall that the set of the extreme points $\cF_n$ of the  unit ball of $\ell_{\infty}$ norm in $\R^n$ are $2^n$ vectors of the form $(\pm 1,\ldots,\pm 1)^\top$.   Hence, 
$\bbf$ is a convex combination of $n+1$ extreme points  in $\cF_n$.  Therefore, 
$$\Sigma_n(F)=\sum_{l=1}^{n+1} g_lD_l \quad g_l\ge 0, l\in[l+1],\sum_{l=1}^{n+1} g_l=1,$$
where $D_l\in\R^{n\times n}$ is a diagonal matrix with diagonal entires $\pm 1$.
Thus, each $D_l$ is unitary, hence each $UD_l V^*$ is unitary.    To prove that $r^\vee(Y)\le 1$, it is enough to show that $r^\vee(\frac{1}{n} (U))\le 1$ for a unitary $U$.
As the spectral decomposition of $U$ is $U=V\diag(t_1,\ldots,t_n)V^*$ we deduce
\begin{equation}\label{specdecU}
U=\sum_{l=1}^n \bv_l t_l\bv_l^*,\quad V=[\bv_1\cdots\bv_n], V^* V=I_n,|t_l|=1, l\in[n].
\end{equation}
Hence, $r^\vee(\frac{1}{n}U)\le 1$.

We now consider the general case $X\succeq 0$ and $\tr X=1$.
Then there exists unitary $U\in \H^{n\times n}$ such that $U^*XU=\Lambda, \Lambda=\diag(\lambda_1, \ldots,\lambda_n)$, where $\lambda_1\ge \cdots\ge \lambda_n\ge 0$, and $\sum_{l=1}^n \lambda_n=1$.   
Observe
$$Z_1:=\diag(U^*, U^*)Z\diag(U,U)=\begin{bmatrix}\Lambda &Y_1,\\Y_1^*&\Lambda\end{bmatrix}, \quad Y_1=U^*YU.$$ 
As  $r^\vee(Y)=r^\vee(Y_1)$,  it suffices to show that $r^\vee(Y_1)\le 1$.
As $Z_1\succeq 0$ we deduce straightforward that if $\lambda_k=0$ then the $k$-the and $n+k$-th row of $Z$ are zero.  Hence, it is enough to consider the case where $\lambda_n>0$.
Let 
\begin{equation*}
\begin{aligned}
&Z_2:=\diag(\Lambda^{-1/2}, \Lambda^{-1/2})Z_1\diag(\Lambda^{-1/2}, \Lambda^{-1/2})=\begin{bmatrix} I_n &F,\\F^*& I_n\end{bmatrix}, \\ 
&F=\Lambda^{-1/2}Y_1\Lambda^{-1/2}.
\end{aligned}
\end{equation*}

Our previous arguments show that  $F$ is a convex combination of $n+1$ unitary matrices.  To conclude the theorem, to is enough to show that $r^\vee(\Lambda^{1/2}U\Lambda^{1/2})\le 1$.  The equaity \eqref{specdecU} yields 
\begin{equation*}
\begin{aligned}
&\Lambda^{1/2}U\Lambda^{1/2}=\sum_{l=1}^n (\Lambda^{1/2}\bv_l)
 t_l(\Lambda^{1/2}\bv_l)^*=\\
& \sum_{l=1}^n \|\Lambda^{1/2}\bv_l\|^2 \bq_lt_l\bq_l^*, \bq_l= \|\Lambda^{1/2}\bv_l\|^{-1}(\Lambda^{1/2}\bv_l), l\in[n].
 \end{aligned}
 \end{equation*}
It is left to show that 
$\sum_{j=1}^n \|\Lambda^{1/2}\bv_j\|^2=1$.
That is 
\begin{equation*}
\begin{aligned}
&\sum_{l=1}^n \tr \bv_l^*\Lambda\bv_l=\Re\sum_{l=1}^n \tr \bv_l^*\Lambda\bv_l=
\Re\sum_{j=1}^n \tr \Lambda\bv_j  \bv_j^*=\\
&\Re\Lambda\tr(\sum_{j=1}^n\bv_j \bv^*_j)=\Re\tr \Lambda=1.
\end{aligned}
\end{equation*}
\end{proof}
\begin{corollary}\label{rveechar}
Let $Y\in\H^{n\times n}$.  Then 
\begin{equation}\label{rveechar1}
r^\vee(Y)=\min\{ \tr W: \begin{bmatrix}W&Y\\Y^*&W\end{bmatrix}\in\rH_{2n,+}(\H)\}.
\end{equation}
\end{corollary}
\begin{lemma}\label{rveelem}  Let $Y\in \H^{n\times n}$.  Then the characterization \eqref{rveechar1} is an SDP characterization of the form \eqref{SDPsf} in $\rH_{2n}(\H)$ with $k=n(2n-1)$.   More precisely,  assume that $W_1,\ldots,W_k$ is a basis in $\rH_n(\H)$, where $W_1=\frac{1}{n}I_n$ and $\tr W_i=0$ for $i=2,\ldots,k$.  Then 
\begin{equation}\label{rveelem1}
X_0=\rH(Y), \, X_i=\diag(W_i,W_i), i\in[k],  c_1=1, c_i=0\, \rm{for}\, i=2,\ldots,k.
\end{equation}
Furthermore,  the strong duality holds, and
\begin{equation}\label{rveelem2}
\begin{aligned}
&r^\vee(Y)=\max\{\Re\tr -H(Y)Z:  \tr Z=n, \\
&\Re\tr  X_iZ=0, i=2,\ldots,k, Z\in\rH_{2n,+}(\H)\}. 
\end{aligned}
\end{equation}
\end{lemma}
\begin{proof}  As $W_1,\ldots,W_{k}$ is a basis in $\rH_n(\H)$ it follows that each $W\in \rH_n(\H)$ has a unique representation $W=\sum_{i=1}^k s_iW_i$.   As $W_1=\frac{1}{n}I_n$ and $\tr W_i=0$ for $i>1$ it follows that $\tr W=s_1$.  
Therefore, 
\begin{equation*}
\begin{aligned}
&Z=X_0+\sum_{i=1}^k s_iX_i=\begin{bmatrix}W&Y\\Y^*&W\end{bmatrix},  \\
&\tr Z=2\tr W=2 s_1,\quad \sum_{i=1}^k c_is_i=c_1 s_1=\tr W
\end{aligned}
\end{equation*}
Therefore, the SDP \eqref{SDPsf} is \eqref{rveechar1}.  Clearly,  if we let $W=(s_1(Y)+1)I_n$ the matrix $Z$ is positive definite.  Hence, Theorem \ref{corSlater} applies.
Observe that \eqref{dSDPsf} is equivalent to \eqref{rveelem2}.
\end{proof}
\subsection{SDP Characterization of $r(\cdot)$}\label{subsec:sdpqr}
The following SDP characterization of qradius  is a generalization of the 
  characterization of $r(C)$  for $C\in\C^{n\times n}$ stated in \cite[Theorem 1.2]{LO20}, which is essentially due to T. Ando [Lemma 1]\cite{And73}.    (See also  \cite[Theorem $2. 1$]{Mat93}).
\begin{theorem}\label{SDPrAq}  Let $A\in\H^{n\times n}$.    Then
\begin{equation}\label{SDPcharA}
r(A)=\min\{a\in\R: \begin{bmatrix} aI_n +Z&A\\A^*& aI_n -Z\end{bmatrix}\in\rH_{n,+}(\H)\}.
\end{equation}
\end{theorem}
\begin{proof} Consider the infimum problem
\begin{equation*}
\mu(A)=\inf\{a\in\R: \begin{bmatrix} aI_n +Z&A\\A^*& aI_n -Z\end{bmatrix}\succeq 0\}.
\end{equation*}
Let $T(a,Z)=\begin{bmatrix} aI_n +Z&A\\A^*& aI_n -Z\end{bmatrix}\in\rH_{2n}(\H)$.
Let $t\in\H, |t|=1$.
Use the equiality \eqref{SAxid} and $T(a,Z)\succeq 0$ to deduce
\begin{equation*}
\begin{aligned}
&aI+Z\succeq 0, \quad aI-Z\succeq 0\Rightarrow a\ge 0, Z\in\rH_n(\H),\\
&(\x^*, -\bar t\x^*)T(a,Z)(\x^*, -\bar t\x^*)^*\ge 0\Rightarrow  2(\x^*\x a-\Re  \x^* A\x t)\ge 0.
\end{aligned}
\end{equation*}
Hence $\mu(A)\ge r(A)$.   Clearly, for $\varepsilon>0$ the following condition hold:
$$T(\|A\|_{\infty}+\varepsilon,0)=(\|A\|_{\infty}+\varepsilon) I_{2n}+H(A)\succ 0.$$ 
Hence $\mu(A)\le \|A\|_{\infty}$.

Observe that the infimum problem for $\mu(A)$ is a standard SDP problem of the form
\eqref{SDPsf}.   Let $k=n(2n-1)+1$ and assume that $W_1,\ldots,W_{k-1}$ is a basis in $\rH_n(\H)$.  Set
\begin{equation}\label{adthmSDPrAq}
\begin{aligned}
&X_0=H(A), X_1=I_{2n},  c_1=1, \\
&X_i=\diag(W_{i-1},-W_{i-1}), c_i=0, i=2,\ldots,k.
\end{aligned}
\end{equation}
Then the infimum problem for $\mu(A)$ is the problem \eqref{SDPsf}.  As we showed that there exists a feasible positive definite matrix, Theorem \ref{corSlater} yields that the value of the dual problem $\mu^\vee(A)$ is equal to $\mu(A)$.  The dual problem
for $\mu(A)$ is given by \eqref{dSDPsf}:
\begin{equation*}
\mu(A)=\max\{\tr H(-A) Z: \tr Z=1, \Re \tr X_i Z=0, =2,\ldots, k, Z\in\rH_{2n}(\H)\}.
\end{equation*}
First observe that the conditions $\tr X_i Z=0, =2,\ldots, k$ yield that $Z=\begin{bmatrix} W&Y\\Y^*&W\end{bmatrix}$, and $\tr W=\frac{1}{2}$.  Theorem \ref{rveeSDPchar}   yields that  $Z\succeq 0$ if and only if $r^\vee(2Y)\le 1$.
Observe that $\tr H(-A)Z=\Re\tr A(-2Y^*)=\Re(-2Y^*)A$.   Recall that $r^\vee (2Y)=r^\vee(-2Y)=r^\vee(-2Y^*)$.  Hence, the dual characterization of $\mu(A)$ is
$\mu(A)=\max\{\Re\tr B^*A, r^\vee(B)\le 1\}$.  Compare that with \eqref{defdnrm1} to deduce that $\mu(A)=r(A)$.
\end{proof}
\subsection{Polynomial computability of $r(\cdot)$ and $r^\vee(\cdot)$}\label{subsec:polcomp}
The following result is a generalization of \cite[Theorem 4.1]{FL23} to quaternions:
\begin{theorem}\label{ptcrA}  Let $A\in \Q^{n\times n}[\bH]$. and $0<\varepsilon\in\Q$.
Then there exists an $\varepsilon$ approximation of $r(A)$ and ,$r^\vee(A),$ in poly-time in 
$n, |\log \varepsilon|$ and the entries of $A$ using the short step primal interior point method combined with  Diophantine approximation.
\end{theorem}
\begin{proof}  
 We can find  $\omega(A)\in\N$ in polynomial time in $n$ and  the entries of $A$   such that $\|A\|_F+1\le \omega(N)\le \|A\|_F+2$.   We first consder $r(A)$.  Recall that $r(A)\le \|A\|_{\infty}\le \|A\|_F$.   We next consider the following subset of selfadjoint matrices  matrices in $\H^{(2n+1)\times (2n+1)}$:
\begin{equation}\label{defcA}
\begin{aligned}
&\cF=\{Y=\begin{bmatrix}aI_n +Z&A&0\\A^*&aI_n-Z&0\\0&0&t\end{bmatrix}\in \rH_{2n+1,+}(\H):\\ 
&2na+t=3(2n+1)\omega(A)\}.
\end{aligned}
\end{equation}
This admissible set $\cF$ has similar description to the admissible set in the proof of Theorem \ref{SDPrAq}.  Let $k=n(2n-1)+1$ and $X_0,\ldots, X_{k}$ be defined as in \eqref{adthmSDPrAq}.  Define
\begin{equation*}
\begin{aligned}
&Y_0=\diag(X_0,3(2n+1)\omega(A)), \\
&Y_1=\diag(X_1,-2n), Y_i=\diag(X_i,0), i=2,\ldots,k.
\end{aligned}
\end{equation*}

It is straightforward to check that the admissible set $\cF$ is of the form $\cA(Y_0,\ldots,Y_k)\cap \rH_{2n+1,+}(\H)$, where we used the notation \eqref{desccA}.
Use Lemma \ref{lcharcAX} to find explicitly $A_j\in \Q^{(2n+1)\times (2n+1)}[\H]\cap \rH_{2n+1}(\H)$ and $b_j\in \Q$ such that
\begin{equation*}
\begin{aligned}
&\cA:=\cA(Y_0,\ldots,Y_k)=\rL(A_1,\ldots,A_m,b_1,\ldots,b_m),  \\
&m=(2n+1)(4n+1)-n(2n-1)-1,  \\
&\dim \rL(A_1,\ldots,A_m,b_1,\ldots,b_m)=n(2n-1)+1.
\end{aligned}
\end{equation*}

Set $F=\frac{1}{2n}\diag(I_{2n},0)$.  It is straightforward to show using \eqref{SDPcharA} that
\begin{equation*}
r(A)=\min\{\langle F,Y\rangle, Y\in \cF\}.
\end{equation*}
It is left to show that the conditions of Theorem \ref{cdeKVal} are satisfied.
Clearly,  we can assume that $n\ge 2$.

Let 
\begin{equation*}
E_0=\begin{bmatrix}3\omega(A)I_n&A&0\\A^*&3\omega(A)I_n&0\\0&0&3\omega(A)\end{bmatrix}=\diag(3\omega(A)I_{2n}+H(A)),3\omega(A)),
\end{equation*}
where $H(A)$ is given in Proposition \ref{SAA}.
 Let $r=\rank A$.  Then $\rank H(A)=2r$, and the nonzero eigenvalues of $(A)$ are $\pm $ of the nonzero singular values of $A$.
Thus 
$$\lambda_{\max}(H(A))=\|A\|_{\infty}\ge \cdots\ge \lambda_{\min}(H(A))=-\|A\|_\infty, \quad \|H(A)\|_{\infty}=\|A\|_{\infty}.$$
Hence, $\lambda_{2n+1}(E_0)> 2\omega(A)$. In particular, $E_0$ is positive definite, and $E_0\in\cF$.  We next show that $\cF\supset \cA\cap \rB(E_0,\omega(A))$.  

Assume that $Y\in \cA\cap \rB(E_0,\omega(A))$.  So $Y$ is of the form given by \eqref{defcA}.  Hence 
\begin{equation*}
\begin{aligned}
&Y-E_0=\diag((a-3\omega(A))I_n+Z,(a-3\omega(A))I_n-Z, t-3\omega(A),  \\
&2n a+t=3(2n+1)\omega(A), \quad \|Y-E_0\|_F\le \omega(A).
\end{aligned}
\end{equation*}
Recall that for any $G\in\H^{p\times q}$ one has inequality $\|G\|_{\infty}\le \|G\|_F$.
Therefore one has the inequalities  
\begin{equation*}
\begin{aligned}
&\omega(A)\ge |t-3\omega(A)|\Rightarrow t\ge 2\omega(A), \\
&\omega^2(A)\ge \|(a-3\omega(A))I_n+Z\|_F^2 +\|(a-3\omega(A))I_n-Z\|_F^2=\\
&2(n(a-3\omega(A))^2+\|Z\|_F^2)\Rightarrow \\
&\omega \ge \sqrt{2n}|a-3\omega(A)|\ge 2|a-3\omega(A)|\Rightarrow a\ge \frac{5}{2}\omega(A),\\
&\omega(A)\ge \sqrt{2}\|Z\|_F\ge \sqrt{2}\|Z\|= \sqrt{2}\|\diag(Z,-Z)\|\ge -\sqrt{2}\lambda_{\min}(\diag(Z,-Z)).
\end{aligned}
\end{equation*}
Hence,
\begin{equation*}
\begin{aligned}
&\lambda_{\min}(aI_{2n}+\diag(Z,-Z)+H(A))=a+\lambda_{\min}(\diag(Z,-Z)+H(A))\ge\\
&a-\frac{1}{\sqrt{2}}\omega(A)-\|A\|_{\infty}> \frac{3}{4}\omega(A)\Rightarrow \lambda_{\min}(Y)> \frac{3}{4}\omega(A).
\end{aligned}
\end{equation*}

We claim that $\cF\subset \cA\cap \rB(E_0, 8n\omega(A))$.  Assume that $Y\in\cF$.  So $Y\succeq 0$ is of the form given by $\eqref{defcA}$.  Hence $aI_n+Z\succeq 0, aI_n-Z\succeq 0, t\ge 0$.  As $2na +t=3(2n+1)\omega(A)$ we deduce that 
\begin{equation*}
\begin{aligned}
&0\le t\le 3(2n+1)\omega(A), \quad 0\le a\le \frac{3(2n+1)}{2n}\omega(A)\le \frac{15}{4}\omega(A)<4\omega(A), \\
&\|Z\|_{\infty}\le a\le 4\omega(A)\Rightarrow \|Z\|_F^2\le 16\omega^2(A)n.
\end{aligned}
\end{equation*}
Hence,
\begin{equation*}
\begin{aligned}
&\|Y-E_0\|_F^2=\|(a-3\omega(A))I_n+Z\|_F^2+\|(a-3\omega(A)I_n-Z\|_F^2 \\
&+(t-3\omega(A))^2
=2n(a-3\omega(A))^2+2\|Z\|_F^2+|t-3\omega(A)|^2\le\\
&(18n+32n+36n^2)\omega^2(A)<64n^2\omega^2(A).
\end{aligned}
\end{equation*}
Observe that $\frac{R}{r}=\frac{8n\omega(A)}{\omega(A)}=8n$.  Use Theorem \ref{cdeKVal} to conclude the proof for $r(A)$.

Consider now $r^\vee(A)$.  Let 
\begin{equation*}
\cA=\{Z=\begin{bmatrix}X&A&0\\A^*&X&0\\0&0&t&\end{bmatrix}\in \rH_{2n+1}(\H): \tr Z=(4n+2)\omega(A)\}.
\end{equation*}
Let $k=n(2n-1)$ and assume that $X_1,\ldots,X_k$ are defined as in \eqref{rveelem1}.
Define teh following matrices in $\in\rH_{2n+1}(\H)$:
\begin{equation*}
\begin{aligned}
&Y_0=\diag(H(A), (4n+2)\omega(A)), Y_1=\diag(X_1, -2), \\ 
&Y_i =\diag(X_i,0)), i=2,\ldots,k.
\end{aligned}
\end{equation*}
It is straightforward to show that $Y_0,\ldots,Y_k$ are linearly independent  The definition \eqref{desccA} yield  that $\cA=\cA(Y_0,\ldots,Y_k)$.  Use Lemma \ref{lcharcAX} to find explicitly $A_j\in \Q^{(2n+1)\times (2n+1)}[\H]\cap \rH_{2n+1}(\H)$ and $b_j\in \Q$ such that
\begin{equation*}
\begin{aligned}
\cA=\rL(A_1,\ldots,A_m,b_1,\ldots,b_m),  \\
m=(2n+1)(4n+1)-n(2n-1)  \\
\dim \rL(A_1,\ldots,A_m,b_1,\ldots,b_m)=n(2n-1).
\end{aligned}
\end{equation*}
Then $\cF=\cA\cap \rH_{2n+1,+}$.  The characterization \eqref{rveechar1} yields that 
$$r^\vee(A)=\min_{Z\in \cF} \tr \diag(I_{n},0)Z.$$

It is left to show that the conditions of Theorem \ref{cdeKVal} are satisfied.
Clearly,  we can assume that $n\ge 2$.
Let $Z_0=\diag(2\omega(A)I_{2n}+H(A), 2\omega(A))$.  Clearly,  $\lambda_{\min}(Z_0)\ge \omega(A)$.  Hence, $Z_0\in \cF$.  The arguments for the case $r(A)$ yield that 
$\rB(Z_0, \omega(A))\cap \cA\subset \cF$.  Similarly, it follows that $\cF\subset \cA\cap \rB(Z_0,5n\omega(A))$.  Hence $\frac{R}{r}=5n$. 
Use Theorem \ref{cdeKVal} to conclude the proof for $r^\vee(A)$.
\end{proof}

Theorem \ref{ptcrA} and the equality \eqref{numrform} yield:
\begin{corollary}\label{polcompmacC1234} Let $C_l\in\rS_{4n}(\R),l\in[4]$ be defined by \eqref{defC1234}.   Then
\begin{equation}\label{maxlmC}
\mu(C_1,\ldots,C_4):=\max_{\|\bt\|\le 1}\lambda_{\max}(\sum_{l=1}^4 t_l C_l)
\end{equation}
is a solution of an SDP problem on $\rS_{4(2n+1)}(\R)$.  Suppose furhtermore that $C_l$ has rational entries for $l\in[4]$.  Then an $\varepsilon$-approximation of $\mu(C_1,\ldots,C_4)$  can be found in poly-time in the entries of $C_l$ and $|\log\varepsilon|$.
\end{corollary}

It is not known to the author if for every four matrices $C_l\in\rS_n(\R)$ there exists an analog of the above corollary.
\section{A pseudo-numerical range on $\C^{n\times n}$}\label{sec:pnrange}
In this section we introduce the notion of pseudo-numerical range and pseudo-numerical radius, abbreviated as \emph{prange} and \emph{pradius} respectively for $A\in\C^{n\times n}$:
\begin{equation}\label{pnumrr}
\begin{aligned}
&\W_{\pi}(A)=\{\x^\top A \x: \x\in \C^n, \|\x\|=1\}, \\
&r_{\pi}(A)=\max_{\|\x\|=1} |\x^\top A\x|.
\end{aligned}
\end{equation}
The equality \eqref{x*Axiden} yields:
\begin{corollary}\label{pqnrco}  
Assume that $A\in\H^{n\times n}$ is of the form $A=A_1+A_2\bj,  A_1,A_2\in\C^{n\times n}$.   Let $P_1,P_2:\H^{n\times n}\to\C^{n\times n}$ be defined by \eqref{projHmn}.  Then
\begin{equation}\label{pqnrco1}
\rP_1(\W(A))=\W(C(A)), \quad \rP_2(\W(A))=\W_{\pi}(C(-A\bj)).
\end{equation}  
In particular, $\rP_1(\W(A))$ is a convex set.
\end{corollary}

Kippenhahn \cite{Kip51} introduced the notion of the bild: $\rB(A)=\W(A)\cap \C$ for $A\in\mathbb{H}^{n\times n}$.  Clearly $\rB(A)\subseteq \rP_1(\W(A))$.  It is known that co$(\rB(A)) = \W(C(A))$ \cite[Theorem 2]{So19}.   For additional results on $\rB(A)$ and its intersection the upper half plane $\rB^+(A)$ see \cite{STZ94,Zha96,Tho97,Kum19}.

We show that that the properties of $W_{\pi}(A)$ are similar to the properties of $\W(A)$ for quaternionic matrices.
Clearly,  $\x^\top A\x=\x^\top A^\top\x$.  Hence
\begin{equation}\label{prrid}
\W_{\pi}(A)=\W_{\pi}(\frac{1}{2}(A+A^\top)), \quad r_{\pi}(A)=r_{\pi}(\frac{1}{2}(A+A^\top)).
\end{equation}
\begin{lemma}\label{prrlem} Let $A\in\C^{n\times n}$.  Then
\begin{enumerate}[(a)]
\item $\W_{\pi}(A)$ is compact, and may not be convex.
\item $\W_{\pi}(A)=\{0\}$ if and only if $A\in\rA_n(\C)$.
\item $r_{\pi}(A)$ is a norm on $\rS_n(\C)$.
\end{enumerate}
\end{lemma}
\begin{proof}
(a) Clearly,  $\W_{\pi}(A)$ is compact.
Assume that $n=1$.  Then $A=[a]$, and $\W_{\pi}(A)=\{z\in \C, |z|=|a|\}$.  Hence, $\W_{\pi}(A)$ is not convex if $|a|>0$.

\noindent
(b) Recall that $A$ has a unique decomposition as $S+T$, where $S\in\rS_n(\C), T\in \rA_n(C)$.  In view of \eqref{prrid} we deduce that $\W_{\pi}(T)=\{0\}$.  It is left to show that $\W_{\pi}(S)=\{0\}$ if and only if $S=0$.   Suppose that $\W_{\pi}(S)=\{0\}$.
Assume that $S=S_1+S_2\bi$ and $\x=\bu+\bi\bv$, where ,  where  $S_1,S_2\in\rS_n(\R)$ and $\bu,\bv\in \R^n, \bu^\top\bu+\bv^\top\bv=1$.  Then
\begin{equation}\label{Siden}
\begin{aligned}
&\x^\top S\x=\big(\bu^\top S_1\bu -\bv^\top S_1\bv-\bu^\top S_2\bv-\bv^\top S_2\bu\big)+\\
&\bi\big(\bu^\top S_2 \bu-\bv^\top S_2\bv+\bu^\top S_1\bv +\bv^\top S_1 \bu\big)=\\
&[\bu^\top \bv^\top]\begin{bmatrix}S_1&-S_2\\-S_2&-S_1\end{bmatrix}\begin{bmatrix} \bu\\ \bv\end{bmatrix} +\bi [\bu^\top \bv^\top]\begin{bmatrix}S_2&S_1\\S_1&-S_2\end{bmatrix}\begin{bmatrix} \bu\\ \bv\end{bmatrix}.
\end{aligned}
\end{equation}
Observe that the two $(2n)\times(2n)$ matrices appearing in the last row of the above idenitity are real symmetric.  The assumption that $\W_{\pi}(A)=\{0\}$ means that the above two real symmetric  matrices of order $2n$ are zero.  Hence $S=0$. 

\noindent
(c)  Clearly, 
\begin{equation*}
r_{\pi}(zA)=|z|r_{\pi}(A), \, r_{\pi}(A+B)\le r_{\pi}(A)+r_{\pi}(B) \textrm{ for } z\in\C. A,B\in \C^{n\times n}.
\end{equation*}
Hence $r_{\pi}(\cdot)$ is a norm on $\rS_n(\C)$ if and only if $r_{\pi}(S)=0\iff S=0$ for $S\in\rS_n(\C)$.  This is shown in (b).
\end{proof}
\begin{theorem}\label{prangesup}  Let $A\in \C^{n\times n}$, and set $S=\frac{1}{2}(A+A^\top)=S_1+S_2\bi$, where $S_1,S_2\in \rS_n(\R)$.  Denote
\begin{equation*}
\hat S_1=\begin{bmatrix}S_1&-S_2\\-S_2&-S_1\end{bmatrix},  \hat S_2=\begin{bmatrix}S_2&S_1\\S_1&-S_2\end{bmatrix}\in\rS_{2n}(\R).
\end{equation*}
Then
\begin{enumerate}[(a)]
\item The set $\rm{co}(\W_{\pi}(A))$ is a compact convex set in $\C$.  The supporting lines of $\rm{co}(\W_{\pi}(A)$  of the form $\Re e^{-\theta\bi}z=Const$ are:
\begin{equation*}
\Re e^{-\theta\bi}z=\lambda_{\min}(\cos\theta \hat S_1+\sin\theta \hat S_2), \quad 
\Re e^{-\theta\bi}z=\lambda_{\max}(\cos\theta \hat S_1+\sin\theta \hat S_2).
\end{equation*}
That is, every $z=x+\bi y\in \rm{co}(\W_{\pi}(A))$ satisfies the sharp inequalities
\begin{equation}\label{sprangein}
\lambda_{\min}(\cos\theta \hat S_1+\sin\theta \hat S_2)\le \cos\theta x+\sin\theta y\le \lambda_{\max}(\cos\theta \hat S_1+\sin\theta \hat S_2).
\end{equation}
\item The pradius of $A$ is given by 
\begin{equation}\label{pradform}
\begin{aligned}
&r_{\pi}(A)=\max\{\lambda_{\max}(t_1 \hat S_1+t_2\hat S_2): \bt=(t_1,t_2)^\top\in\R^2, \|\bt\|\le1\}=\\
&\max\{\lambda_{\max}(t_1 \hat S_1+t_2\hat S_2): \bt=(t_1,t_2)^\top\in\R^2, \|\bt\|=1\}.
\end{aligned}
\end{equation}
\end{enumerate}
\end{theorem}
\begin{proof}
(a)  The first equality in \eqref{prrid} yields that $\W_{\pi}(A)=\W_{\pi}(S)$.  Thus, without loss of generality we can assume that $A=S$.
As $\W_{\pi}(S)$ is a compact set it follows that co$(\W_{\pi}(S))$ is a compact convex set.    Let $z\in \W_{\pi}(S)$.  Then $z=\x^\top S\x$ for some $\x\in \C^n, \|\x\|=1$.  As in the proof of Lemma \ref{prrlem} let $\x=\bu+\bi\bv, \bu^\top\bu +\bv^\top\bv=1$.
The identity \eqref{Siden} yields that
\begin{equation*}
z=[\bu^\top\bv^\top] \hat S_1\begin{bmatrix}\bu\\\bv\end{bmatrix} +\bi [\bu^\top\bv^\top] \hat S_2\begin{bmatrix}\bu\\\bv\end{bmatrix} 
\end{equation*}
Hence,
\begin{equation*}
\Re e^{-\theta\bi }z=[\bu^\top\bv^\top] \big( \cos\theta \hat S_1+\sin\theta\hat S_2\big)\begin{bmatrix}\bu\\\bv\end{bmatrix}
\end{equation*}
Take the minimum and the maximum of the above expression on $(\bu^\top,\bv^\top)^\top$ with norm one to deduce the sharp inequalities \eqref{sprangein}.

\noindent (c)
Clearly,  for each $(\bu^\top,\bv^\top)^\top$ with norm one has the inequality 
\begin{equation*}
\sqrt{\big([\bu^\top\bv^\top] \hat S_1\begin{bmatrix}\bu\\\bv\end{bmatrix}\big)^2+\big([\bu^\top\bv^\top] \hat S_1\begin{bmatrix}\bu\\\bv\end{bmatrix}\big)^2}\le r_{\pi}(S).
\end{equation*}
Hence $\lambda_{\max}(\cos\theta \hat S_1+\sin\theta \hat S_2)\le r_{\pi}(S)$.  
Observe next that that there exists $\theta\in[0,2\pi)$ and $(\bu^\top,\bv^\top)^\top$ of length one such that 
\begin{equation*}
\Re e^{-\theta\bi }\x^\top S\x=[\bu^\top\bv^\top] \big( \cos\theta \hat S_1+\sin\theta\hat S_2\big)\begin{bmatrix}\bu\\\bv\end{bmatrix}=r_{\pi}(B).
\end{equation*}
Hence, 
\begin{equation*}
r_{\pi}(B)\le \lambda_{\max}( \cos\theta \hat S_1+\sin\theta\hat S_2).
\end{equation*} 
This shows that
\begin{equation}\label{pradform1}
r_{\pi}(B)=\max_{\theta\in[0,2\pi)} \lambda_{\max}( \cos\theta \hat S_1+\sin\theta\hat S_2).
\end{equation} 
As $\lambda_{\max}(t_1\hat S_1+t_22\hat S_2)$ is a convex function of $\bt=(t_1,t_2)^\top\in\R^2$ if follows that the maximum of $\lambda_{\max}(t_1\hat S_1+t_22\hat S_2)$ the unit disk $\|\bt\|\le 1$ achived on the boundary.
Hence, \eqref{pradform1} is equivalent to \eqref{pradform}.
\end{proof}

We now recall \cite[Lemma 5.1]{FL23}:  
\begin{lemma}\label{knwnr(C)}
Let $C=E+F\bi $, where $E,F\in \rH_n$.  Then
\begin{equation*}
r(C)=\max_{\theta\in[0,2\pi)} \lambda_{\max}(\cos\theta E+\sin\theta F).
\end{equation*}
\end{lemma}
\begin{corollary}\label{corthmprangesup} Let the assumptions of Theorem \ref{prangesup} hold.   Set $C=\hat S_1+\hat S_2\bi\in \C^{(2n)\times (2n)}$.  Then 
\begin{equation*}
r_{\pi}(A)=r_{\pi}(S)=r(C).
\end{equation*}
\end{corollary}

It is straightforward to show that $\W_{\pi}(S)\subset \W(C)$.  Is it true that co$(\W_{\pi}(S))=\W(C)$?  For $n=1$ a straighforward calculation shows that one has equality.

The well known result of Ando \cite{And73} implies that $r(C)$ is a solution of an SDP problem.  Assume that $C$ has Gaussian rational entries,  and  $\varepsilon >0$ is rational.    Theorem 3.3 in \cite{FL23} shows the computation of $r(C)$ within precision $\varepsilon$  is polynomially computable in data of the entries of $C$ and $|\log\varepsilon|$.  Hence, same results apply to $A$ with rational Gaussian entries.
\section*{Acknowledgment}
I thank the referee for the remarks and comments.
The author is partially supported by the Collaboration Grant for Mathematicians of the Simons Foundation.
\section*{Dedication}
This paper is dedicated to my friend Avi Berman.


\begin{thebibliography}{MMM}
\bibitem{And73} T. Ando,  Structure of operators with numerical radius one,
Acta Sci. Math. (Szeged) 34 (1973), 11–15.
\bibitem{Bre51} J. Brenner, Matrices of quaternions,  {\it Pacific J. Math. }1 (1951), 329-335.
\bibitem{deklerk_vallentin}
E. de~Klerk and F. Vallentin, On the {T}uring model complexity of interior
  point methods for semidefinite programming, {\it SIAM J. Optim. } 26~(3),
 2016,  1944--1961.
\bibitem{Fri15}  S. Friedland, \emph{Matrices: Algebra, Analysis and Applications}, World Scientific, 596 pp., 2015, Singapore, http://www2.math.uic.edu/$\sim$friedlan/bookm.pdf
\bibitem{FL23} S. Friedland and C.-K. Li, On a semidefinite programming characterizations of the numerical radius and its dual norm,  arXiv:2308.07287.
\bibitem{GM12}  B. G\"artner and J.  Matou\v{s}ek,  Approximation algorithms and semidefinite programming, Springer, Heidelberg, 2012, xii+251 pp.
\bibitem{HJ13} R.A. Horn  and C.R. Johnson, Topics in Matrix Analysis, Cambridge: Cambridge University Press, Second edition  2013.
\bibitem{Kip51}  R. Kippenhahn, {\"U}ber den Wertevorrat einer Matrix, \emph{Math. Nachr.,} 6:193-228, 1951,  English Translation: On the numerical range of a matrix,  translated from German by P.F.  Zachlin and M. E.  Hochstenbach,  {\it Linear and Multilinear Algebra} 56(1-2):185-225.
\bibitem{Kum19} P. Kumar, A note on convexity of sections of quaternionic numerical range,  {\it Linear Algebra and its Applications,} 572 (2019), 92-116.
\bibitem{Lee49} H.C. Lee, Eigenvalues and canonical forms of matrices with quaternion coefficients, {\it Proc. Royal
Irish Acad. } Sect. A 52 (1949), 253-260.
\bibitem{LO20} A.S. Lewis and M.L. Overton,
Partial Smoothness of the Numerical Radius at Matrices whose Fields of Values are Disks, {\it SIAM J. Matrix Anal. Appl. }41 (2020), pp. 1004–1032.
\bibitem{Mat93} R. Mathias, Matrix completions, norms and Hadamard products, {\it Proc. Amer.  Math. Soc.}, 117(4):905-918, 1993.
\bibitem{Mit23} T.  Mitchell, Convergence rate analysis and improved iterations for numerical radius computation,  SIAM J. Sci. Comput. 45 (2023), no. 2, A753-A780.
\bibitem{MO23}T. Mitchell and M. L. Overton, An experimental comparison of methods for computing the numerical radius, in preparation.
\bibitem{So19} W. So,  The early development of the quaternionic numerical range,  {\it IMAGE},  {\it Bulletin of the International Linear Algebra Society,} 63 (2019), 7-11. 
\bibitem{STZ94} W. So, R.C. Thompson,  F.Z. Zhang,  The numerical range of normal matrices with quaternion entries, {\it Linear and Multilinear Algebra} 37 (1994), no. 1-3, 175–195.
\bibitem{ST96} W. So, R.C. Thompson, Convexity of the upper complex plane part of the numerical range of a quaternionic matrix, {\it Linear Multilinear Algebra} 41 (1996), 303-365.
\bibitem{Tho97} R. C. Thompson, The upper numerical range of a quaternionic matrix is not a complex numerical range,
{\it Linear Algebra Appl.}, 254 (1997), 19-28, .
\bibitem{VB96} L. Vandenberghe and S. Boyd, Semidefinite Programming, SIAM Review 38, March 1996, pp. 49-95.
\bibitem{Zha96} F. Zhang,On numerical range of normal matrices of quaternions, {\it J. Math. Phys. Sci.,} 29(6) (1995), 235-251.
\bibitem{Zha97} F. Zhang,  Quaternions and matrices of quaternions,
{\it Linear Algebra Appl. } 251 (1997), 21-57.
 \end{thebibliography}
\end{document}